\documentclass[12pt]{amsart} 
 \usepackage{amssymb, latexsym, amsmath} 
 \usepackage{amssymb, latexsym, amsmath}

\begin{document} 

 \theoremstyle{plain} 
 \newtheorem{theorem}{Theorem}[section] 
 \newtheorem{lemma}[theorem]{Lemma} 
 \newtheorem{corollary}[theorem]{Corollary} 
 \newtheorem{proposition}[theorem]{Proposition} 

\theoremstyle{definition} 
\newtheorem*{definition}{Definition}
\newtheorem{example}[theorem]{Example}
\newtheorem{remark}[theorem]{Remark}

\title[Bounds on ideals]{Bounds on the number of ideals in finite 
  commutative nilpotent $\Fp$-algebras}
	

\author{Lindsay N. Childs}
\email{lchilds@albany.edu}
\author{Cornelius Greither}
\email{cornelius.greither@unibw.de}

\date{\today}

\newcommand{\QQ}{\mathbb{Q}} 
\newcommand{\FF}{\mathbb{F}} 
\newcommand{\ZZ}{\mathbb{Z}}
\newcommand{\ZZm}{\mathbb{Z}/m\mathbb{Z}}
\newcommand{\ZZp}{\mathbb{Z}/p\mathbb{Z}}
\newcommand{\ee}{\end{eqnarray}} 
\newcommand{\ben}{\begin{eqnarray*}} 
\newcommand{\een}{\end{eqnarray*}} 
\newcommand{\dis}{\displaystyle} 
\newcommand{\beal}{\[ \begin{aligned}} 
\newcommand{\eeal}{ \end{aligned} \]} 
\newcommand{\lb}{\lambda} 
\newcommand{\Gm}{\Gamma} 
\newcommand{\bpm}{\begin{pmatrix}} 
\newcommand{\epm}{\end{pmatrix}} 
\newcommand{\Fp}{\mathbb{F}_p} 
\newcommand{\Fpx}{\mathbb{F}_p^{\times}} 
\newcommand{\Ann}{\text{Ann} }
\newcommand{\tb}{\textbullet \ }
\newcommand{\GL}{\mathrm{GL}}
\newcommand{\gb}{\genfrac{[}{]}{0pt}{}}
\newcommand{\M}{\mathrm{M}}
\newcommand{\Aut}{\mathrm{Aut}}
\newcommand{\End}{\mathrm{End}}
\newcommand{\Perm}{\mathrm{Perm}}
\newcommand{\PGL}{\mathrm{PGL}}
\newcommand{\diag}{\mathrm{diag}}
\newcommand{\Reg}{\mathrm{Reg}}

\parskip=5pt plus5pt

\begin{abstract} Let $A$ be a finite commutative nilpotent $\Fp$-algebra structure 
on $G$, an elementary abelian group of order $p^n$.  If $K/k$ is a Galois extension 
of fields with Galois group $G$ and $A^p = 0$, then corresponding to $A$ is an 
$H$-Hopf Galois structure on $K/k$ of type $G$.  For  that Hopf Galois structure 
we may study the image of the Galois correspondence from $k$-subHopf algebras 
of $H$ to subfields of $K$ containing $k$ by utilizing the fact 
that the intermediate subfields correspond to the $\Fp$-subspaces of $A$, while 
the subHopf algebras of $H$ correspond to the ideals of $A$.  We obtain upper 
and lower bounds on the proportion of subspaces of $A$ that are ideals of $A$, and test the bounds on some examples.
\end{abstract}

\maketitle


\section*{Introduction}

The motivation for this work is to understand the Galois correspondence for 
certain Hopf Galois structures on field extensions. 

Let $K/k$ be a Galois extension of fields with Galois group $G$.  Then the Galois correspondence sending 
subgroups $G'$ of $G$ to subfields $K^{G'}$ of $K$ containing $k$ is, by the Fundamental Theorem of Galois Theory,  a bijective correspondence from subgroups of $G$ onto the intermediate fields between $k$ and $K$.    

In 1969 S. Chase and M. Sweedler \cite{CS69} defined the concept of a Hopf Galois 
extension of fields for a field extension $K/k$ and $H$ a $k$-Hopf algebra 
acting on $K$ as an $H$-module algebra.  They proved a weak version of the 
FTGT, namely, that there is an injective  Galois correspondence from 
$k$-subHopf algebras $H'$ of $H$ to intermediate fields, given by 
$H' \mapsto K^{H'}$, the subfield of elements fixed under the action of $H'$. 
But surjectivity was not obtained.  Greither and Pareigis \cite{GP87} defined 
a class of non-classical Hopf Galois structures, the "almost classical" 
structures, for which surjectivity holds, but also  gave an example 
where it fails.   Recent work of Crespo, Rio and Vela  (\cite{CRV15} and 
especially \cite{CRV16}) studied the image of the Galois correspondence for 
Hopf Galois structures on separable extensions $K/k$ with normal closure 
$\tilde{K}$ and found numerous examples where surjectivity fails.  
In nearly all of  the cases examined in \cite{CRV16} the Galois group of  
$\tilde{K}/K$ is non-abelian.  

In this paper we seek to quantify the failure of the FTGT for Hopf Galois 
structures of the following type.

Let $K/k$ be a Galois extension of fields with Galois group $G$, an elementary 
abelian $p$-group of order $p^n$. Suppose $H$ is a $k$-Hopf algebra of 
type $G$ (that means, $K \otimes_k H \cong KG$), and $K/k$ is a $H$-Hopf 
Galois extension.  As shown in \cite{Ch15}, \cite{Ch16}, 
\cite{Ch17}, building on work of 
\cite{CDVS06} and \cite{FCC12}, every $H$-Hopf Galois structure  of type $G$ on a 
Galois extension of fields $K/k$ with Galois group $G$, an elementary 
abelian $p$-group, arises from a commutative nilpotent $\Fp$-algebra 
structure $A$ on the additive group $G$ with $A^p = 0$.   In \cite{Ch17},  
it was shown that the sub-$K$-Hopf algebras of $H$ correspond to ideals 
of $A$.  For a Galois extension $K/k$ whose Galois group is an elementary 
abelian $p$-group (or equivalently, an $\Fp$-vector space), the classical 
FTGT gives a bijection between $\Fp$-subspaces of $G$ and intermediate fields.   So let $i(A)$ denote the number of ideals of $A$, and $s(A)$ the number of $\Fp$- subspaces of $A$.  Then the proportion of intermediate fields $k \subseteq E\subseteq K$ that are in the image of the Galois correspondence  for a $H$-Hopf Galois structure on $K/k$ arising from $A$ is equal to $i(A)/s(A)$.

As observed in \cite{Ch17}, that comparison  implies immediately that if $A^2 \ne 0$, then there are subspaces of $A$ that are not ideals, and hence the Galois correspondence cannot be surjective.   

Let $e$ be the unique integer such that $A^e \ne 0$ and
$A^{e+1} = 0$; we assume throughout that $e>0$ (that is, $A$ is not zero)
and $e < p$.  To quantify the failure of surjectivity of the FTGT for a Hopf Galois 
structure corresponding to $A$, we obtain in section 2 of this paper a general 
upper bound, depending only on $e$, on the ratio $i(A)/s(A)$. The upper bound implies, for example, that for $e \ge 3$ and $p \ge 17$, $i(A)/s(A) < 0.01$.  

Using information on the dimensions of the annihilator ideals of $A$, we obtain 
in section 3 a lower bound on $i(A)$.  

The upper bound is based on a lower bound on the fibers of the ``ideal generated by'' function $G$ from subspaces of $A$ to ideals of $A$.  In the final section we examine that lower bound on fibers of $G$, and the inequalities of sections 2 and 3, for some examples.  

Let $s(n)$ denote the number of subspaces of  an $\Fp$-vector space of 
dimension $n$.  Then $s(n)$ is a sum of Gaussian binomial coefficients, 
also called $q$-binomial coefficients (where $q = p$).  The first section 
of the paper describes properties of these coefficients and obtains 
inequalities relating $s(m)$ and $s(n)$ for $m < n$.  
  
Throughout the  paper, we assume that  $A$ has dimension $n$ and that $A^p = 0$.  
Recall that $e$ is the largest number so that $A^e \ne 0$ (so $A^{e+1} = 0$).  
All vector spaces are over $\Fp$.  

Our thanks go to the University of Nebraska at Omaha and to Griff Elder 
for their hospitality and support.


\section{Gaussian binomial coefficients}

To compare the number of ideals of a commutative nilpotent $\Fp$-algebra $A$ 
with the number of subspaces of $A$, we need to collect some information 
concerning the number of subspaces of dimension $k$ of an $\Fp$-vector space 
of dimension $n$.  So we begin with Gaussian binomial coefficients.  

The Gaussian binomial coefficient, or $q$-binomial coefficient (here $q = p$), 
is defined 
as
\beal\gb{n}{k} 
&= \frac{(p^n-1)(p^n-p) \cdots (p^n-p^{k-1})}{(p^k-1)(p^k-p) \cdots (p^k-p^{k-1})}\\
&= \frac{(p^n-1)(p^{n-1} -1) \cdots (p^{n-(k-1)} -1)}{(p-1)(p^2-1) \cdots (p^k- 1)}.
\eeal
It counts the number of $k$-dimensional subspaces of $\Fp^n$.  So 
\[ \gb{n}{k} = \gb{n}{n-k} \text{ for all }k,\]
$\gb{n}{ 0}= \gb{n}{n}=1$, and $\gb{n}{ k} = 0$ for $k > n$.  Then
\[s(n) = \sum_{k=0}^n \gb{n}{k}\]
is the total number  of subspaces of $\Fp^n$.    
Note that  it suffices to replace the factors $(p^n- p^r)/(p^k-p^r)$ by $p^n/p^k$ 
in order to see that
\[ \gb{n}{k}  \ge p^{k(n-k)},\]
and that $\gb{n}{k}$ has order of magnitude $p^{k(n-k)}$ for large $p$.  

(In fact, the rational function 
   \[\gb{n}{k}_x = \frac{(x^n-1)(x^n-x) \cdots (x^n-x^{k-1})}{(x^k-1)(x^k-x) 
   \cdots (x^k-x^{k-1})}\]
is a polynomial of degree $(n-k)k$ in $\ZZ[x]$.  For let $b(x), a(x)$ be the 
numerator and denominator of $\gb{n}{k}_x$.  Both are monic polynomials 
in $\ZZ[x]$.  Dividing $b(x)$ by $a(x)$ in $\QQ[x]$ gives
\[ b(x) = a(x)q(x) + r(x),\]
where $\deg(r(x)) < \deg(a(x))$. Since $a(x)$ is monic, $q(x)$ and $r(x)$ are 
in $\ZZ[x]$. Now $b(p)/a(p) = \gb{n}{k}_p$ is a positive integer for every 
prime $p$, so the rational function $r(p)/a(p)$ is also an integer for every 
prime $p$.  But 
\[\lim_{p \to \infty} \frac {r(p)}{a(p)} = 0.\]
So $r(p) = 0$ for all primes greater than some fixed bound, and hence $r(x) = 0$.  
So $b(x)/a(x) = q(x)$ is in $\ZZ[x]$.) 

The Gaussian binomial coefficients satisfy two recursive formulas, analogous 
to that satisfied by the usual binomial coefficients:
\beal \gb{n}{k} &= \gb{n-1}{k-1} + p^k\gb{n-1}{k}\\
&= \gb{n-1}{k}+ p^{n-k}\gb{n-1}{k-1}.\eeal

Using properties of the Gaussian binomial coefficients, we will now obtain 
some inequalities relating the number of subspaces of $\Fp$-vector spaces 
of dimensions $n, n-1$ and $n-2$ for all $n$.  

Let $\delta(n) = \lfloor \frac {n^2}{4}\rfloor = \begin{cases}
  n^2/4 &\text{ if $n$ is even}\\
  (n^2 -1)/4 &\text{ if $n$ is odd}.
\end{cases} $ 

\bigskip

\begin{lemma}\label{4.1}
\noindent
 \begin{itemize}
   \item[a)] For all $n\ge 2$ we have $s(n) \ge p^{n-1}s(n-2)$.
	  \vskip2pt
	 \item[b)] If $n>1$ is even, then $s(n) \ge \frac{1}{2}p^{n/2}s(n-1)$.
	\vskip2pt
	 \item[c)] If $n>0$ is odd, then $s(n) \ge p^{(n-1)/2}s(n-1)$.
	\vskip2pt
	 \item[d)] For $n\ge m\ge 0$ arbitrary, we have $s(n) \ge 
	     \frac{1}{2}p^{\delta(n)-\delta(m)} s(m)$.
	      The factor $1/2$ may be omitted if $m$ and $n$ have 
				the same parity or if $n$ is even.
 \end{itemize}
\end{lemma}

\begin{proof} a) Using the two recursion formulas for $\gb{n}{d}$ in turn we find:
\begin{eqnarray*}
  \gb{n}{d} &=&  \gb{n-1}{d-1} +  p^d \gb{n-1}{d} \\
            &=&  \gb{n-1}{d-1} +  p^d(p^{n-1-d} \gb{n-2}{d-1}  +  \gb{n-2}{d})  \\
            &\ge&                p^{n-1}\gb{n-2}{d-1} .
\end{eqnarray*}
Summing these for $d=1,\ldots, n-1$ gives the required inequality. 

b) Let $n=2k$. We may calculate as follows:
\begin{eqnarray*}
  s(n) &\ge&  \gb{n}{1} + \ldots + \gb{n}{k} \\
	   &=& (p^{n-1}\gb{n-1}{0}+\gb{n-1}{1}) + (p^{n-2}\gb{n-1}{1}+\gb{n-1}{2}) + \ldots \\
		 & &  \quad \ldots +   (p^{k}\gb{n-1}{k-1}+\gb{n-1}{k})   \\
		 &\ge & p^{k}\gb{n-1}{0} + p^{k}\gb{n-1}{1} + \ldots p^{k}\gb{n-1}{k-1} \\
		 &=& p^k s(n-1)/2. 
\end{eqnarray*}

c) Let $n=2k+1$. 
Using one recursive formula, then the other, we get:
\beal  
\gb{n}{1} &\ge p^{n-1}\gb{n-1}{0} \ge p^k \gb{n-1}{0}; \\
\gb{n}{2} &\ge p^{n-2}\gb{n-1}{1} \ge p^k \gb{n-1}{1}; \\
& \vdots\\
\gb{n}{k-1} &\ge p^{n-(k-1)}\gb{n-1}{k-2} \ge p^k \gb{n-1}{k-2}; 
\eeal
  (now we switch to the other recursive formula)
\beal
\gb{n}{k} &\ge p^{k}\gb{n-1}{k}; \\
\gb{n}{k+1} &\ge p^{k+1}\gb{n-1}{k+1}; \\
\gb{n}{k+2} &\ge p^{k+2}\gb{n-1}{k+2}\ge p^k\gb{n-1}{k+2}; \\
& \vdots\\
\gb{n}{n-1} &\ge p^{n-1}\gb{n-1}{n-1} \ge p^k \gb{n-1}{n-1}.
\eeal
Now observe that
\[ \gb{n-1}{k+1} = \gb{2k}{k+1} = \gb{2k}{k-1}.\]
Therefore
\[ p^{k+1}\gb{n-1}{k+1} = p^k(2\gb{n-1}{k+1}) = p^k(\gb{n-1}{k+1} + \gb{n-1}{k-1}).\]
Thus $s(n)$ is at least as large as the sum of the left sides of the inequalities, 
which is at least the sum of the right sides of the inequalities, and in view 
of the last observation, the sum of the right sides is at least $p^k s(n-1)$.  

d) We first note that $\delta(n) -  \delta(n-2) = n-1$, so by a),
\[ s(n) \ge p^{n-1}s(n-2) = p^{\delta(n) - \delta(n-2)}s(n-2).\]
Iterating this shows that if $n > m$ and $n \equiv m \pmod{2}$ then
\[ s(n) \ge p^{\delta(n) - \delta(m)}s(m).\]
If $n$ is even and $m$ is odd, then $n/2 = \delta(n) - \delta(n-1)$ and 
by b) we find
\[ s(n) \ge \frac 12 p^{\frac n2}s(n-1)  = \frac 12 p^{\delta(n) - \delta(n-1)}s(n-1),\]
so 
\[ s(n) \ge \frac 12 p^{\delta(n) - \delta(m)}s(m).\]
If $n$ is odd and $m$ is even, then $(n-1)/2 = \delta(n) - \delta(n-1)$, so by c), 
\[s(n) \ge p^{\frac{n-1}2}s(n-1) = p^{\delta(n) - \delta(n-1)}s(n-1),\]
hence
\[ s(n) \ge  p^{\delta(n) - \delta(m)}s(m).\]
 \end{proof}


\section{An upper bound on the number of ideals of $A$}

In this section we obtain a general upper bound for the ratio $i(A)/s(A)$
of  the number of ideals of $A$ to the number of
subspaces of $A$, for $A$ an arbitrary commutative 
nilpotent $\Fp$-algebra of dimension $n$.  To do so, we consider the function 
$G$ from subspaces of $A$ to ideals of $A$ which associates to each subspace $U$
the ideal $G(U)=U+AU$ generated by $U$, and we establish a lower bound 
on the cardinality of the fiber of each ideal under this map (which is 
obviously surjective).  But first, we need to count subspaces with certain 
properties. 
  
Recall that $\delta(t)=\lfloor t^2/4\rfloor$. All vector spaces are over 
$\mathbb F_p$, and the number of $k$-dimensional subspaces of $W$, an 
$\Fp$-vector space of dimension $d$, is $\gb{d}{k}$.  

  We show:

\begin{proposition}\label{2.1}
Let  $\dim(W) = d$  and let $W_0$ be a fixed subspace of $W$ of dimension $r$.    
For $k \le d-r$, the number $s(d, r; k)$  of $k$-dimensional subspaces $U$ 
of $W$ with $U\cap W_0 = (0)$ is equal to $p^{rk}$ times the number of 
$k$-dimensional subspaces of $W/W_0$:
\[  s(d, r; k) = p^{rk}\gb{d-r}{k}.\]
 \end{proposition}
 
\begin{proof} Let $V$ be a complementary subspace to $W_0$, so that 
$V \oplus W_0 = W$.   For $k \le d-r$, let $U$ be a $k$-dimensional 
subspace of $V$, with basis $(z_1, \ldots, z_k)$.  For each choice 
$\overline{a} = (a_1, \ldots, a_k)$ of elements of $W_0$,  the subspace 
$U_{\overline{a}}$ of $W$ generated by  $(z_1+ a_1, \ldots, z_k + a_k)$ 
is $k$-dimensional and has trivial intersection with $W_0$.  For  suppose 
\[ c_1(z_1 + a_1) + \ldots c_k(z_k + a_k) = a \]
in $W_0$ for some $c_1, \ldots, c_k$ in $\Fp$.  Then, since $V \oplus W_0$ 
is a direct sum of $\Fp$-vector spaces, 
\[ c_1z_1  + \ldots c_kz_k  = 0.\]
Since $z_1, \ldots z_k$ are linearly independent,  $c_1, \ldots c_k = 0$, 
hence $a = 0$.  The same argument with $a = 0$ shows that $(z_1+ a_1, \ldots, 
z_k + a_k)$  is a linearly independent set.

Finally, each choice of elements $\overline{a} =(a_1, \ldots, a_k)$ of $W_0$ 
gives a different subspace $U_{\overline{a}}$ of $A$.  For suppose $z_i + b_i$ 
is in the space $U_{\overline{a}}$.  Then
\[ z_i + b_i = c_1(z_1 + a_1) + \ldots  + c_i(z_i + a_i) + \ldots + c_k(z_k + a_k) .\]
So 
\[ 0 = c_1(z_1 + a_1) + \ldots + ((c_i -1)z_i + c_ia_i - b_i) + \ldots + c_k(z_k + a_k). \]
But then
\[ 0 = c_1z_1 + \ldots + (c_i -1)z_i + \ldots + c_kz_k .  \]
So $c_i = 1$, all other $c_j = 0$,  and the equation reduces to
\[ a_i -b_i = 0.\]
Thus for each $k$-dimensional subspace $U$ of $W$, we obtain $p^{rk}$ $k$-dimensional subspaces $U_{\overline{a}}$ of $W$ with $W \cap W_0= (0)$. 
\end{proof}

\begin{corollary}\label{2.2}  Let $W$ be a $t$-dimensional space and $W_0\subset W$
a subspace of codimension 1. Then the number of subspaces
of $W$ not contained in $W_0$ is at least $p^{\delta(t)}$.
  \end{corollary}

\begin{proof}  First we remark that via a duality argument, the
number of subspaces of dimension $k$ 
not contained in a fixed subspace of codimension 1 is
the same as the number of subspaces of dimension
$t-k$ intersecting a fixed subspace of dimension 1
trivially. Hence the preceding proposition is applicable; summing over all
possible dimensions of $U$, we find that the number of subspaces $U\subset W$
not contained in $W_0$ is
\beal  s(t, 1) &{:=} \sum_{k = 0}^{t-1} s(t, 1; k)\\
&= \sum_{k = 0}^{t-1} p^k \gb{t-1}{k}\\
& \ge \sum_{k = 0}^{t-1} p^k p^{k(t-1-k)}\\
& = \sum_{k = 0}^{t-1}  p^{k(t-k)}\\
&\ge p^{\lfloor \frac {t^2}4 \rfloor} = p^{\delta(t)}. \eeal
\end{proof}

Recall that $G$ is the map  from subspaces of $A$ to ideals of $A$ defined by
\[ G(V) = V + AV. \]
To simplify notation, we write $G(S)$ instead of $G(\langle S \rangle_{\Fp})$
for any subset $S$ of $A$.    To get a sense of the relationship between the number 
of subspaces of $A$ and the number of ideals of $A$, we  will count the number of 
elements in the fibers of $G$.

Assume $e>0$ is minimal with $A^{e+1} = 0$. (The zero algebra $A=0$ can be
safely excluded from our study.) Consider the chain 
\[ N_1 \subset N_2 \subset \ldots \subset N_e = A \]
of annihilator ideals defined by
\[ N_k := \Ann_k(A) = \{a \in A| x_1x_2 \cdots x_ka = 0 \text{ for all } x_1, 
 \dots x_k \text{ in } A\}.\]
Let $\dim_{\Fp}(N_k) = d_k$.  Then the sequence $(d_k)_k$ is obviously
increasing, and a little argument shows that $0 < d_1 < d_2 < \ldots < d_e = n$.

The strategy for bounding the number of ideals of $A$ begins with the following idea.  
Let $\mathcal{J}_t$ be the set of ideals $J$ of $A$ contained in $N_t$ but not contained in $N_{t-1}$.  
Then, since $N_t$ is an ideal of $A$ for all $t$, we have
\[ \sum_{J \in \mathcal{J}_t} |G^{-1}(J)| = s(N_t) - s(N_{t-1}) .\]
The  next lemma will help us find a lower bound on $|G^{-1}(J)|$.  
\begin{lemma} \label{3.1} Let $W = G(\{x\}) =\Fp x + Ax$ , $W_0 = Ax$ as above.  
Let $U$ be a subspace  of $W$, not contained in $W_0$.  Then $G(U) = W$. 
\end{lemma}

\begin{proof}
Let $y$ be in $U$, $y$ not in $W_0$.  Then $G(\{y\}) \subseteq G(U)$.  
After multiplying $y$ by a non-zero element of $\Fp$,  we can assume that 
$y = x - ax$ for some $a$ in $A$.  Then
\[ y + ay + a^2y + \ldots + a^{e-1}y = x\]
is in $G(\{y\})$.  So
\[ W = G(\{x\}) \subseteq G(\{y\}) \subseteq G(U) \subseteq W .\]
\end{proof}

Let $J$ be an ideal  of $A$ of  $\Fp$-dimension $d$, let $s(J)$ (or $s(d)$) 
be the number of subspaces of $J$, and let $i(J)$  be the number of ideals 
of $A$ that are contained in $J$.  Lemma \ref{3.1}  enables us to prove 
a result relating the number of subspaces and the number of ideals contained in 
the annihilator ideal $N_t$ in $A$ for each $t$.

\begin{proposition}\label{2.4} 
For each $t$ with $1 \le t \le e$, consider the ideal map $G$ restricted 
to the set of subspaces $V$ of $N_t$ that are not contained in $N_{t-1}$.  
For each $x$ in $N_t \setminus N_{t-1}$, let $q(x) = \dim(G({x}))$, and 
let $q_t = \min_{x \in N_t \setminus N_{t-1}}q(x)$.  Then for all ideals $J$ in $\mathcal{J}_t$, 
\[ |G^{-1}(J)| \ge p^{\delta(q_t)}.\]
  Hence
\[ p^{\delta(q_t)}\bigl(i(N_{t}) - i(N_{t-1})\bigr) \le s(N_{t}) - s(N_{t-1}).\]
\end{proposition}

\begin{proof} 
Let $J$ be an ideal contained in $N_t$, not contained in $N_{t-1}$.  
Let $x$ be in $J$, $x$ not in $N_{t-1}$.  Let $W_0 = Ax$ and $W = G(\{x\}) 
 = \FF_px + Ax$.  Then $W$ has dimension at least $q_t$, and $W_0$ has 
codimension 1 in $W$.  Let $Y$ be a complement of $W$ in  $J$.  Then for 
every subspace $U$ of $W$ not contained in $W_0$, we have $G(U)=W$ 
and thus $G(U+Y)=J$. 

Whenever $U$ and $U'$ are distinct subspaces of $W$ not contained in $W_0$, 
we have $U + Y \ne U' + Y$.  So the number of preimages of $J = G(\{x\} + Y)$ 
is at least equal to the number of subspaces of $W$ that are not contained 
in $W_0$.  Since $\dim(W) \ge q_t$, that number of subspaces is 
$\ge p^{\delta(q_t)}$ by Corollary \ref{2.2}.  
\end{proof}

Dividing both sides of the $t$-th inequality of Proposition \ref{2.4} 
by $p^{\delta(q_t)}$ and summing them over all $t$ yields an upper bound 
for the number of ideals of $A$:

\begin{corollary}\label{3.2}
\[ i(A) \le \sum_{t=1}^{e-1} (p^{-\delta(q_t)}-p^{\delta(q_{t+1})})s(N_t) 
 + p^{-\delta(q_e)}s(N_e). \]
 \end{corollary}

Omitting the negative terms and applying Lemma \ref{4.1} d) yields  the 
following upper bound on $i(A)$ in terms of $s(A)$  (recall 
$d_t = \dim{N_t}$):
 
\begin{corollary}\label{3.2bis} 
  \[ i(A) \le \bigl(\sum_{t=1}^{e-1} 2p^{-\delta(q_t) +\delta(d_t) - \delta(d_e)}
 + p^{-\delta(q_e)}\bigr)s(A). \]
\end{corollary}
 
To make it easier to apply this inequality for general $A$, we show
the following simple lower bound on the quantity $q_t$. (Recall it was defined by
$q_t = \min_{x \in N_t \setminus N_{t-1}}q(x)$ with $q(x) = \dim(G({x}))$.)

\begin{proposition}  
 For all $t>0$ we have $q_t \ge t$.  
\end{proposition}

\begin{proof}   This is clear for $t = 1$.  

 For $t > 1$ let  $x$ be in $N_t$ and not in $N_{t-1}$.  Let 
$u_1, u_2, \ldots, u_{t-1}$ in $A$ so that $u_1u_2 \cdots u_{t-1}x \ne 0$.  
Then for each $k$, $x_k = u_{k}\cdots u_{t-1}x$ is in $N_{k}$ and not 
in $N_{k-1}$.  So $x_1, \ldots, x_{t-1}, x$ are linearly independent in $A$.  
Thus $G({x}) = \Fp x + Ax$ has dimension at least $t$.  \end{proof}

In the next theorem we will use this lower bound on $q_t$
to get a general, fairly elegant upper bound on $i(A)/s(A)$ that only depends
on $e$, the length of the annihilator chain in $A$.  However, in some 
of the examples treated below it will be worthwhile to have
a closer look at $q_t$; we will find it to be considerably
larger than $t$, which will enable us to sharpen the upper bound. 

The general bound goes as follows.

\begin{theorem}\label{main} With the above hypotheses on $A$ and $e$ we have
\[   \frac{i(A)}{s(A)}  \le \frac{2e-1}{p^{\delta(e)}}. \]
\end{theorem}

\begin{proof}
In the inequality of Corollary \ref{3.2bis}, replace $q_t$  by $t$ and 
observe that since $1<d_1<d_2<\ldots< d_e$, one has 
$\delta(d_e)-\delta(d_t)\ge \delta(e)-\delta(t)$.
If we insert this into the inequality, the terms $p^{\pm\delta(t)}$
cancel and we obtain
\[  i(A) \le \sum_{t=1}^{e-1} 2p^{-\delta(e)}s(N_e) + p^{-\delta(e)}s(N_e)
 = (2e-1)p^{-\delta(e)} s(A). \] \end{proof}

 For $e = 2, 3$  the inequalities of Theorem \ref{main} are
\beal  i(A) &\le \frac 3p s(A)   \text{  \quad     for }e = 2;\\
i(A) &\le \frac 5{p^2} s(A)   \text{  \quad     for }e = 3.
\eeal
We can improve these bounds by some constant factors, 
(almost) without imposing further conditions
on the algebra $A$.   Recall that $n = \dim(A)$.

\begin{proposition} \label{main3} For $e = 2$, we have
\[  i(A) \le \frac 2p s(A)\]
whenever $p \ge 3$ and n $\ge 3$.  For $e = 3$, we have
\[ i(A) \le \frac  2{p^2}s(A)\]
whenever $p \ge 3, n \ge 4$. \end{proposition}

\begin{proof}  Case $e = 2$:   From Corollary \ref{3.2} with 
$\delta(q_t)$ replaced by $\delta(t)$, we have
\[ i(A) \le  \bigl(1 - \frac 1p\bigr)s(N_1) +\frac 1p s(A). \]
To get the claimed inequality it suffices to assume that $\dim{N_1} = n-1$ 
and show that
\[ \bigl(1 - \frac 1p\bigr)s(n-1) \le \frac 1p s(n),\]
or $(p-1)s(n-1) \le s(n)$.  Using Lemma \ref{4.1}b) for $n$ even it suffices 
to show that
\[ \frac {p^{n/2}}2 > p-1,\]
which holds for $p \ge 3, n\ge 4$, while for $n$ odd, it suffices 
by Lemma \ref{4.1}c) to show that 
\[ p^{n/2} > p-1,\]
which holds for $p \ge 3, n \ge 3$.  

Case $e = 3$.  From Corollary \ref{3.2} we have
\[ i(A) \le \bigl(1 - \frac 1p\bigr)s(N_1) + 
   \bigl(\frac 1p - \frac 1{p^2}\bigr)s(N_2) +  
	  \frac 1{p^2}s(A).\]
Since $s(A) = s(n)$, the right side is maximized when 
$s(A) = s(n), s(N_2) = s(n-1), s(N_1) = s(n-2)$.  
To show that the right side is $\le \frac 2{p^2}s(A)$, it suffices 
to show that
\[ \bigl(1 - \frac 1p\bigr)s(n-2) + \bigl(\frac 1p - \frac 1{p^2}\bigr)s(n-1) 
 \le \frac 1{p^2} s(n).\]
Using Lemma \ref{4.1}b) for $n$ even, we are reduced to showing that
\[ \bigl(\frac 1p - \frac 1{p^2}\bigr) \frac 2{p^{n/2}} +  
   \bigl(1 - \frac 1p\bigr)\frac 1{p^{n-1}} 
	 \le \frac 1{p^2},\]
which holds for $p \ge 3, n \ge 4$.  Using Lemma \ref{4.1}c) for $n$ odd, 
we see it suffices to show that 
\[ \bigl(1 - \frac 1p\bigr)\frac 1{p^{n-1}}  + 
 \bigl(\frac 1p - \frac 1{p^2}\bigr) \frac 1{p^{(n-1)/2}}  
   \le \frac 1{p^2},\]
which holds for $n \ge 5$ and $p \ge 2$.  \end{proof}

 The bounds of Theorem \ref{main} and Proposition \ref{main3} imply: 
 
\begin{corollary}\label{2.10} Suppose $K/k$ is a Galois extension with elementary abelian $p$-group $G$ and is also a $H$-Hopf Galois extension  where $H$ arises from a commutative nilpotent $\Fp$-algebra structure $A$ on the additive group $G$,  where $A^e \ne 0, A^{e+1} = 0$ and $e < p$.  Then  $i(A)/s(A)$ is the proportion of intermediate fields that are in the image of the Galois correspondence from sub-Hopf algebras of $H$, and $i(A)/s(A) < 0.01$ for
 
 \tb $e =2, p \ge 200$,
 
 \tb $e = 3, p \ge 17$,
 
 \tb $e = 4, p \ge 7$,
 
 \tb all $e, p$ with $5 \le e < p$.
 \end{corollary}


\section{A lower bound on the number of ideals}

We now obtain a lower bound  on the number of ideals of $A$, 
by exhibiting a collection of ideals in $A$ and estimating its size.    
Recall that $A^e\not=0 = A^{e+1}$ and that we defined
\[ N_r = \Ann_r(A) = \{a \in A | x_1x_2 \cdots x_ra = 0 \text{  for all } x_1,
 \ldots, x_r \text{ in } A\}.\]
Then $N_r$ is an ideal of $A$, and 
\[ (0) \subset N_1 \subset N_2 \subset \ldots \subset N_e = A,\]
all inclusions being proper.

We already defined $d_r = \dim(N_r)$;  let us put  
$t_r = \dim_{\Fp}(N_{r}/N_{r-1})$.  
For each $r=1,\ldots,e$, let $W_r$ be a subspace of $A$ so that 
\[ N_{r} = W_r \oplus N_{r-1}.\]
(In particular, $W_1= N_1$.)  
Then $t_r = \dim(W_r)$,
\[  A = W_1 \oplus W_2 \oplus \ldots \oplus W_e\]
and 
\[ t_1 + t_2 + \ldots + t_e = n.\]

\begin{proposition}\label{lower}
   \[ i(A) \ge  \lb(A) := s(t_1) + (s(t_2) -1) +
         \ldots + (s(t_e)-1).\]
\end{proposition}

\begin{proof}  
  For each $r$, $1 \le r\le e$, and each non-zero subspace $V_r$ of $W_r$, 
let $J = N_{r-1} + V_r$.  Then $J$ is an ideal of $A$.  
Indeed, we have $AV_r \subset AN_r \subset N_{r-1}$,
and therefore $AJ \subset AN_{r-1}+N_{r-1} \subset N_{r-1} \subset J$.		
	
The formula $\lb(A)$ of the proposition simply counts the number 
of ideals $J$ just described. \end{proof}  

Since for any $m$,
 \[ s(m) = \sum_{k=0}^m \gb{m}{k}\]
 and $\gb{m}{k} \ge p^{(m-k)k}$, we can let $t_M = \max{t_k}$ and get 
a rough lower bound for the number of ideals in $A$:
 \[ i(A) \ge p^{\delta(t_M)}.\]

 
\section{Some classes of examples}

To see how sharp the bounds on ideals are that we obtained in the last two sections, we look at some explicit classes of algebras. 

\begin{example}  First, consider the  ``uniserial'' $e$-dimensional
algebra $A$ generated by $x$ with $x^{e+1}=0$.  In this case,  for every element 
$u$ in $N_t \setminus N_{t-1}$, the dimension $q(x) = \dim(G(\{x\})$ is
equal to $t$.   We then see that the general upper bound 
\[ i(A) \le (2e-1)p^{-\delta(e)}s(A)\]
is in fact close to the true number $i(A)=e+1$ for large $p$, 
since  $s(A)$ is a polynomial in $p$ of degree $\delta(e)$. 

The lower bound $\lb(A)$ in this simple class of examples is $e+1$.  
\end{example}

\begin{example} \label{binom4}
Let $A$ be a ``binomial'' nilpotent algebra:  
$A = \langle x_1, x_2, \ldots x_ e \rangle$ with $x_k^2 = 0$ for all $k$.  
Then $\dim(A) = 2^e -1$, $\Ann(A) = (x_1x_2\cdots x_e)$ and 
$\dim(N_t/N_{t-1}) = \binom e{t-1}$.
 
 Theorem 2.6 tells us that the ratio of ideals to subspaces for $A$ is 
bounded as follows:
 \[ \frac{i(A)}{s(A)} \le \frac {2e-1}{p^{\delta(e)}} 
  = \frac {2e-1}{p^{\lfloor \frac {e^2}4 \rfloor}}.\]
But that inequality arose from minorizing  
$q_t = \min_{x \in N_t \setminus N_{t-1}} \dim(G(x))$ by $t$ throughout.
In this class of examples we can do better, having a closer
look at $q_t$.

\begin{proposition} \label{4.3} Let $A$ be the binomial algebra  of dimension $2^e -1$.  
Then for every non-zero $u$ in $N_t \setminus N_{t-1}$ we have
 \[ \dim(G(u)) \ge 2^{t-1}.\]
\end{proposition}

\begin{proof}  For any given
$u$ in $N_t \setminus N_{t-1}$, pick a monomial summand $y$ of $u$ in 
$N_t \setminus N_{t-1}$.  Renumber the variables of $A$ so that
 \[ y = x_1x_2 \cdots x_{e-t+1}.\]     
Then we introduce an ordering on the set
of all nonzero monomials of $A$ so that any monomial of $N_{k-1}$ 
comes after any monomial of $N_k\setminus N_{k-1}$ for all $k$,  
and the monomials within $N_k \setminus N_{k-1}$ are ordered lexicographically.
Strictly speaking, this is a total ordering on the set of all monomials
up to multiplication with a nonzero scalar in $\Fp$. 

Call a family of monomials admissible if no two of them are equal
up to a nonzero scalar. Every nonzero $z\in A$ has a unique ``leading''
monomial $m(z)$, according to the ordering. The following is easy
to see: if $(z_i)_{i\in I}$ is a family of elements of $A$, such
that the family of leading monomials $(m(z_i))_i$ is admissible,
then $(z_i)_i$ is $\Fp$-linearly independent. If $w$ is any monomial,
we have $m(zw)=m(z)w$.

Now consider the family $F$ of monomials that consist only of factors
$x_{e-t+2},\ldots, x_e$; this family has $2^{t-1}$ entries, and is
of course admissible. If we multiply every element of this family
by $u$, the leading terms just get multiplied by the monomial $y$,
so they again are an admissible family. Hence the entries of the
family $uF$ are again linearly independent, which shows that the
ideal $G(u)$ generated by $u$ has dimension at least $2^{t-1}$.
   \end{proof}

We illustrate how working with  $q_t \ge 2^{t-1}$ instead of 
the crude lower bound $q_t\ge t$ affects the upper bound 
on the ratio $i(A)/s(A)$ of Theorem \ref{main} for a binomial algebra.

Consider the binomial algebra $A = \langle x_1, x_2, x_3, x_4 \rangle$ 
with $x_i^2 = 0$.  Then 
\[ \dim(N_1) = 1, dim(N_2) = 5, \dim(N_3) = 11, \dim(N_4) = 2^4-1 = 15.\]
(Note $N_4=A$.) The general inequality \ref{main} gives
\[ \frac {i(A)}{s(A)} \le \frac 7{p^4}.\]
Let us start afresh. From Corollary \ref{3.2} we have
\[ i(A) \le \sum_{t=1}^{3} (p^{-\delta(q_t)}-p^{\delta(q_{t+1})})s(N_t) 
 + p^{-\delta(4)}s(A). \]
Omitting the negative terms gives
\[ i(A) \le \frac 1{p^{\delta(q_1)}}s(N_1) + \frac 1{p^{\delta(q_2)}}s(N_2)  
 + \frac 1{p^{\delta(q_3)}}s(N_3)  + \frac 1{p^{\delta(q_4)}}s(N_4) .\]
Now $\delta(q_1) = \delta(1) = 0$ and for $t > 1$, $\delta(q_t) \ge
\delta(2^{t-1}) = 2^{2t-4}$. So we have
\[ i(A) \le s(1) + \frac 1p s(5)  + \frac 1{p^4} s(11)  + \frac 1{p^{16}}s(15). \]
Now we use Lemma \ref{4.1} d):
\beal s(15) &\ge p^{\delta(15) - \delta(1)} s(1) = p^{56} s(1);\\
s(15) & \ge  p^{\delta(15) - \delta(5)} s(5) = p^{50} s(5);\\
s(15) &\ge p^{\delta(15) - \delta(11)} s(11)= p^{26} s(11).
\eeal
So
\[ \frac {i(A)}{s(A)} \le (\frac 1{p^{56}} + \frac 1{p^{51}} 
  +\frac 1{p^{30}} +\frac 1{p^{16}}) \le \frac 2{p^{16}}.\]
This is a big improvement over the inequality above that comes from the general
approach.

However, the lower bound $\lb(A)$ on the number of ideals of $A$ from Proposition 
\ref{lower} is a polynomial in $p$ of degree 9, while 
\[ \frac 2{p^{16}}s(15) > \frac 2{p^{16}}\gb{15}{7}  > \frac 2{p^{16}} p^{56} = 2p^{40}.\] 
So there remains a large gap between the upper and lower bounds on $i(A)$.
\end{example}
In general, the gap between the upper and lower bounds for $i(A)$ arises because 
the upper bound is based on a lower bound on the sizes of fibers of the ideal generator 
map 
\[G: (\text{ subspaces of }A )\to \text{ (ideals of  }A).\]
For $J$ an ideal of $N_k$, not in $N_{k-1}$, we showed that $|G^{-1}|(J)| 
\ge p^{\delta(q_k)}$ where $q_k$ is the minimum of the dimensions of \emph{principal} 
ideals $G(x)$ for $x$ in $N_k\setminus N_{k-1}$.  But for many nilpotent algebras 
$A$  and many ideals $J$ of $A$, this lower bound greatly underestimates $|G^{-1}(J)|$.  
We illustrate this with two examples.

\begin{example}
Let $A$ be the ``triangular'' algebra $A = \langle x, y \rangle$ with $A^{e+1} = 0$.  
Here one sees that  $q_u = \frac {t(t+1)}2$ for $u$ in $N_t \setminus N_{t-1}$.    
Let us look at the case $e = 2$ in detail.

Let $A = \langle x, y \rangle$ with $A^3 = 0$.  
Then $A$ has a basis $\mathcal{B} = (x, y, x^2, xy, y^2)$, and the annihilator 
$N_1=\Ann(A)$ has basis $x^2, xy, y^2$. Moreover $N_2=A$.  
So we have $d_1=3$ and $d_2=5$.

\begin{proposition}\label{idcount}
There are $3p^2 + 4p + 6$ ideals in $A$.  
\end{proposition}

\begin{proof} The lower bound $\lb(A)$ from Proposition \ref{lower} counts ideals of 
$N_1$ and ideals of $A$ properly containing $N_1$:  that number is 
\beal  \lb(A) &= s(t_1) + s(t_2) -1 \\&= s(3) + s(2) -1\\
&= (2p^2 + 2p + 4) + (p+2) \\&= 2p^2 + 3p + 6.\eeal
To determine the number of ideals of $A$ we let $\bar A = A/N_1$. This is 
the two-dimensional algebra spanned
by $\bar x$ and $\bar y$ with zero multiplication. We classify ideals $J\subset A$
by their image $\bar J$ in $\bar A$. Those with $\bar J=0$ are simply the
subspaces of $N_1$, which we've already counted.  
One easily sees that $\bar J=\bar A$ only happens once, for $J=A$, and since that ideal contains $N_1$, it is already counted.  

There remains the case where $\bar J$ is one-dimensional. There are
$p+1$ one-dimensional subspaces of $\bar A$, but by applying suitable
automorphisms of $A$ it suffices to count ideals with $\bar J=\Fp \bar x$, 
and multiply that count by $p+1$. All such $J$ contain $x^2$ and $xy$, so the
question is whether they contain $y^2$. If yes, $J$ is simply the
linear span of $x$ and $N_1$ and has been counted. If no, then $J$
contains an element $x+ay^2$ for a unique scalar $a\in \Fp$;
this scalar determines the ideal.  So there are $p$ such ideals mapping onto 
$\Fp \bar x$.   Thus the count of ideals $J$ with $\bar J$ one-dimensional is $p(p+1)$.  Adding that number to $L(A)$ gives the result.
\end{proof}

We can write down all of the ideals explicitly and determine their fibers.  
The notation $(m)$ denotes ``subspace generated by $m$''.  In the list, 
$a, b, d$ are arbitrary elements of $\Fp$.

\beal &A  = G(x, y)\\
& J_1 = J_1(a, d) = G(x + ay + dy^2)  = (x + ay + dy^2, x^2 -a^2y^2, xy + ay^2) \\
&J_{15} = J_{15}(a) = G( x + ay, y^2) = (x + ay, x^2, xy, y^2) \\
&J_2 = J_2(b) = G(y + bx^2) = (y + bx^2, xy, y^2)\\
&J_{23} = G(y, x^2) = (y, x^2, xy, y^2)\\
&\text{and finally all subspaces of  the ideal }N_1 = (x^2, xy, y^2). \eeal

To describe the subspaces of $A$, choose the basis $(x, y, x^2, xy, y^2)$ of $A$.  
Looking at row vectors of coordinates with respect to that basis yields a bijection 
between subspaces of $A$ and row spaces of  $5 \times 5$ matrices 
with entries in $\Fp$.  Those row spaces are in bijective correspondence with the set of 
$5 \times 5$ reduced row echelon matrices.  Those, in turn, can be categorized 
by specifying the columns where the pivots occur: the number of pivots specified 
defines the dimension of the subspace.  Thus the label (124) denotes the 
$5 \times 5$ reduced row echelon matrix
\[ \bpm 1&0& \cdot & 0 & \cdot\\0&1& \cdot & 1& \cdot\\0&0&0& 1 & \cdot \epm\]
(we omit all rows of zeros),  where the five unspecified entries can be 
arbitrary elements of $\Fp$.  Thus there are $p^5$ subspaces of $A$ corresponding 
to echelon forms with label (124).

The echelon forms (3), (4), (5), (34), (35), (45), (345) define the non-zero subspaces 
of $N_1$. Those subspaces are also ideals of $A$ since multiplication on $N_1 = Ann(A)$ is trivial.  

Every echelon form that includes both 1 and 2 defines a subspace of $A$ that generates the ideal $A$.   It is possible to discuss all other forms in turn, finding the ideals
generated by the corresponding subspaces and the exact size of the fiber of $G$.
Since this is repetitive and space-consuming, we only write out what happens
for three echelon forms. 

(1)  has the form $(1, a', b', c', d')$ and generates $J_1(a, d)$ for $a=a'$ and 
$d = d' + b'a'^2 -c'a'$.  So for each $(a, d)$ there are $p^2$ subspaces  of 
type (1) that generate $J_1(a, d)$.
 
(13)  has the form $\bpm 1& a' &0 &c' & d'\\0&0&1&e'&f'\epm$.  If $a' = a$, 
$d' -c'a +e'a^2 =  d$ and $f '= e'a -a^2$ , then it generates $J_1(a, d)$.    
In that case,  for each $(a, d)$ there are $p^2$ subspaces of type (13) that 
generate $J_1(a, d)$.  
If $a' = a$ and $f '\ne e'a -a^2$, then the subspace generates $J_{15}(a)$.  
In that case the choices for $(c', d', e', f')$ yield $p^3(p-1)$  subspaces 
of type (13) that generate $J_{15}(a)$.

(14) has the form  $\bpm 1& a' &b' &0 & d'\\0&0&0&1&f'\epm$. If $f' = a' = a$ 
then the subspace generates $J_1(a,  d)$ for all $p$ choices of $b'$; otherwise 
for $a' = a \ne f'$ there are $p^2(p-1)$ choices of $(b', d', f')$ for subspaces 
of type (14) that generate $J_{15}(a)$.

As noted, we omit the (easy) discussion of the remaining forms (15), (134), (135),
(2), (23), (234), (235), (2345),  (24), (25), (245).

Adding up the number of subspaces that generate each ideal, we get the tables below.  
Here $a, b, c$ are arbitrary elements of $\Fp$.
 \begin{center} \begin{tabular}
{c|c|c}
ideals &\# of ideals & fiber size\\ \hline
$A$&$ 1 $&$ 2p^6 + p^5 +2p^4 + p^3 + p^2 + 1$\\
$J_1(a, d) $&$p^2$&$2p^2 + p + 1  $\\ 
$J_{15}(a) $&$p$&$ p^4 + p^3  + p^2 + 1 $\\ 
$J_2(b) $&$p$&$2p^2 + p + 1  $\\ 
$J_{23}  $&$1$&$ p^4 + p^3 + p^2 + 1 $\\
ideal of $N_1$ &$2p^2 + 2p + 4$&$ 1 $\\ 
\end{tabular}\end{center}
The center column sums to the number of ideals of $A$.

The total number of subspaces of $A$ accounted for by fibers of ideals of each type is:
 \begin{center} \begin{tabular}
{c|c}
ideals & \# subspaces\\\hline
$A$&$2p^6 + p^5 +2p^4 + p^3 + p^2 + 1$\\
$J_1(a, c)$&$2p^4 + p^3 + p^2  $\\ 
$J_{15}(a) $&$ p^5 + p^4  + p^3 + p $\\ 
$J_2(b) $&$2p^3 + p^2 + p  $\\ 
$J_{23} $&$ p^4 + p^3 + p^2 + 1 $\\
ideal of $N_1$ &$2p^2 + 2p + 4$\\ 
\end{tabular}\end{center}
The right column sums to $s(5)=$ the number of subspaces of $A$.

Let us compare this with our more general results.  We have
\[ s(5)=2p^6+2p^5+6p^4+6p^3+6p^2+4p+6.\]
Given that $i(A)=3p^2+4p+6$ by Prop.~\ref{idcount},
the inequality of Proposition \ref{main3}  comes out as 
\[ 3p^2+4p+6 \le 4p^5+4p^4+12p^3+12p^2+12p+8+12p^{-1}. \]
This inequality was based on assuming that every  ideal $J$ not contained in 
$N_1$ has dimension $\ge 2$, and so $|G^{-1}(J)| \ge p^{\delta(2)} = p$ . 

In Corollary \ref{3.2bis}, the factor in brackets between $\le$
and $s(A)$ evaluates to $2p^{-4}+p^{-2}$, using $d_1=3, d_2=5$
and $q_1=1$, $q_2=3$.  This assumed that  every ideal $J$  not contained in 
$N_1$ has dimension $\ge 3$, so  $|G^{-1}(J)| \ge p^{\delta(3)} = p^2$.  
Then the inequality is
\[  3p^2+4p+6 \le 2p^4+2p^3+10p^2+10p+18 + r(p),  \]
where $r(p)=12p^{-4}+8p^{-3}+18p^{-2}+16p^{-1}$ is always positive
but tends to 0 for $p\to\infty$.

Looking at the actual sizes of the fibers of $G$ in this example, the 
inequality $|G^{-1}(J)| \ge p^2$ has the correct power of $p$ for principal 
ideals $J$.  But the non-principal ideals $J_{23}$, $J_{15}(a)$ and $A$ that are not contained in $N_2$ have 
fibers with cardinalities of order $p^4, p^4$ and $p^6$, respectively.  
This helps explain why the general upper bound on ideals is loose.    
\end{example}

\begin{example} Let $A = \langle x, y, z \rangle$ with $x_i^2 = 0$, the binomial algebra in three variables.    Then $A^4 = 0$  ($e = 3$) and 
$A$ has a basis 
\[ (x, y, z, xy, xz, yz, xyz)\]
with $N_1 = (xyz), N_2 = ( xy, xz, yz, xyz)$.  
The number of subspaces of $A$ is 
\[ s(7) = 2p^{12} + 2p^{11} + 6p^{10} + 8p^9 + \text{  terms in }p\text{ of lower degree}.\]
The number of ideals of $A$ turns out to be 
\[ i(A) = 7p^2 + 4p + 8.\] 
 The lower bound  on $i(A)$, the number of ideals of $A$, is 
\[ \lb(A) = s(1) +(s(3)-1) + (s(3) -1)  = 4p^2 + 4p + 8.\]
An  upper bound on $i(A)$ can be obtained by using Proposition \ref{4.3}, which says that the dimension of a principal ideal of $A$ not contained in $N_2$ is at least 4.  Then we get 

\beal
i(A) &\le (\frac 1{p^{\delta(7) - \delta(1)}} + \frac 2{p \cdot p^{\delta(7) -\delta(4)}} + \frac 1{p^4})s(A)\\
&= (\frac 1{p^{12}} + \frac 2{p^9} + \frac 1{p^4})s(A)\\
&\le \frac 2{p^4} s(A) \sim 4p^8 + 2p^7 + \ldots .\eeal
To see why this upper bound on $i(A)$ is off by a factor of a constant times $p^6$, we can determine $|G^{-1}|$ for the ideals of $A$, by methods in the last example.  We omit the details.  But we observe first that the fibers of the $2p^2 + 3p + 4$ ideals of $N_2$ account in total for $s(4) = p^4+ 3p^3 + 4p^2 + 3p + 5$ subspaces of $A$.  So most subspaces of $A$ generate ideals not contained in $N_2$.  

In obtaining our upper bound, we used that for principal ideals of $A$ not contained in $N_2$, $|G^{-1}(J)| \ge p^4$.   But in fact, we find that:

\tb  For the $p^2 + p + 1$ principal ideals $J$ of the form $G(x + by+ cz)$ with $bc \ne 0$ in $\Fp$, the dimension of $J \ge p^5$, so $|G^{-1}(J)| \ge p^{\delta(5)} = p^6$, not $p^4$.  Thus this set of ideals is generated by approximately $p^8$ subspaces of $A$. 

\tb  For the $p^2 + p+ 1$ non-principal ideals of the form
\[ G(x+bz, y + cz), G(x + by, z), G(y, z),\]
each is generated by at least $p^9$ subspaces of $A$.  Thus this set of ideals is generated by approximately $p^{11}$ subspaces of $A$. 
 
\tb Finally, for the ideal $A = G(x, y, z)$ itself, every subspace of $A$ whose reduced row echelon form has the form $(123\ldots)$ generates $A$, and summing the number of such subspaces yields 
\[|G^{-1}(A)| \ge 2p^{12} + p^{11} + 2p^{10} + 2p^9 + \ldots .\] 

Comparing that to $s(7) = s(A)$ above, it is evident that the weakness in the upper bound we found for $i(A)$ arises from the considerable underestimation of the size of $G^{-1}(J)$ for non-principal ideals  not contained in $N_2$, and, in particular, on the size of $G^{-1}(A)$:  $|G^{-1}(A)|$ is a polynomial in $p$ of the same degree as $s(A)$.

This last fact turns out to be true in general.   One can show (proof omitted) that $|G^{-1}(A)|$
is always a polynomial in $p$ with the same degree as the polynomial $s(A)$,
under the fairly mild assumption that $A$ as an $\Fp$-algebra is generated
by at most $\dim(A)/2$ elements. 
\end{example}

From these examples it appears that any substantial tightening of 
the upper bound for the ideals of $A$ will require a more nuanced look 
at the fibers of non-principal ideals whose $\Fp$-dimension is close to the dimension of $A$. 

However, the primary objective of this paper has been achieved.  
Let  $L/K$ be a Galois extension  with elementary abelian Galois group 
an elementary abelian $p$ group $G$.  If $L/K$ is a $H$-Hopf Galois 
extension of type $G$ corresponding to a commutative nilpotent algebra 
structure $A$ on $G$ with $A^p = 0$,  then the upper bound on $i(A)/s(A)$ in section 2, weak as it may be for some examples, still provides the first general quantitative estimate on how far from surjective is the Galois correspondence for the Hopf Galois structure on $L/K$.


\end{document}